\documentclass[12pt]{amsart}
\usepackage{a4wide}
\usepackage{enumerate}
\usepackage{graphicx}
\usepackage{color}
\usepackage{amsmath}
\usepackage{scrextend} 
\allowdisplaybreaks

\let\pa\partial
\let\na\nabla
\let\eps\varepsilon

\newcommand{\R}{{\mathbb R}}
\newcommand{\diver}{\operatorname{div}}

\newtheorem{theorem}{Theorem}

\newtheorem{remark}[theorem]{Remark}
\newtheorem{corollary}[theorem]{Corollary}

\begin{document}

\title[Uniqueness of weak solutions to cross-diffusion systems]{A note on the
uniqueness of weak solutions to a class of cross-diffusion systems}

\author[X. Chen]{Xiuqing Chen}
\address{School of Sciences, Beijing University of Posts and Telecommunications,
Beijing 100876, China}
\email{buptxchen@yahoo.com}

\author[A. J\"ungel]{Ansgar J\"ungel}
\address{Institute for Analysis and Scientific Computing, Vienna University of
	Technology, Wiedner Hauptstra\ss e 8--10, 1040 Wien, Austria}
\email{juengel@tuwien.ac.at}

\date{\today}

\thanks{The first author acknowledges support from the National Natural Science
Foundation of China (NSFC), grant 11471050, and from the China Scholarship Council
(CSC), file no.\ 201706475001, who financed his stay in Vienna.
The second author acknowledges partial support from
the Austrian Science Fund (FWF), grants P27352, P30000, F65, and W1245}

\begin{abstract}
The uniqueness of bounded weak solutions to strongly coupled parabolic equations
in a bounded domain with no-flux boundary conditions is shown. The equations
include cross-diffusion and drift terms and are coupled selfconsistently
to the Poisson equation.
The model class contains special cases of the Maxwell-Stefan equations
for gas mixtures, generalized Shigesada-Kawasaki-Teramoto equations for
population dynamics, and volume-filling models for ion transport.
The uniqueness proof is based on a combination of the
$H^{-1}$ technique and the entropy method of Gajewski.
\end{abstract}



\maketitle


\section{Introduction}

Several techniques have been developed for the analysis of nonlinear parabolic
systems, including sufficient conditions for the global existence of weak
or strong solutions \cite{Ama89,Jue16,LSU68,Pie10}. However, the proof of uniqueness
of weak solutions is generally much more delicate, in particular for strongly
coupled systems. In this paper, we prove the uniqueness of bounded weak solutions
to a class of cross-diffusion systems.
The proof is based on a combination of the $H^{-1}$ technique and the
method of Gajewski \cite{Gaj94}, where a certain semimetric measures the
distance between two solutions. It is shown that the semimetric is related
to relative entropies.

\subsection{Model equations}

The equations describe the evolution of the concentrations $u_i$,
\begin{equation}\label{1.eq1}
  \pa_t u_i = \diver\sum_{j=1}^n\big(A_{ij}(u)\na u_j + B_{ij}(u)\na\phi\big),
	\quad i=1,\ldots,n,
\end{equation}
in a bounded domain $\Omega\subset\R^d$ ($d\ge 1$),
where $u=(u_1,\ldots,u_n)$ and $\phi$ is a potential solving the Poisson equation
\begin{equation}\label{1.eq2}
  -\Delta\phi = u_0-f(x)\quad\mbox{in }\Omega,
\end{equation}
where $u_0=\sum_{i=1}^n a_iu_i$ for some constants $a_i\ge 0$,
and $f(x)$ is a given background concentration.
We complement the equations by no-flux boundary and initial conditions,
\begin{equation}\label{1.bic}
  \sum_{j=1}^nA_{ij}(u)\na u_j\cdot\nu = \na\phi\cdot\nu = 0\ \mbox{on }\pa\Omega,
	\quad u(0)=u^0\mbox{ in }\Omega, \quad i=1,\ldots,n.
\end{equation}
For consistency, the initial datum has to satisfy the condition
$$
  \int_\Omega \sum_{i=1}^n a_iu_i^0 dx = \int_\Omega f(x)dx.
$$

The diffusion coefficients $A_{ij}$ and drift coefficients $B_{ij}$ are defined by
\begin{equation}\label{1.AB}
  A_{ij}(u) = p(u_0)\delta_{ij} + a_ju_i q(u_0), \quad
	B_{ij}(u) = r(u_0)u_i\delta_{ij}, \quad i,j=1,\ldots,n,
\end{equation}
for some functions $p$, $q$, and $r$ and numbers $a_j\ge 0$.
Our main assumption is that these functions do not depend on the species number $i$.
Then $u_0$ satisfies a nonlinear drift-diffusion equation (see \eqref{2.eq1} below),
and this property allows us to initiate the uniqueness proof. We do not know how
to relax this assumption in the context of weak solutions.

The diffusion matrix $A(u)=(A_{ij}(u))$ is {\em not} assumed to be positive definite
and it may degenerate. The existence theory developed in \cite{Jue15} is based
on the assumption that there exists a transformation of variables such that the
transformed diffusion matrix becomes positive semidefinite, allowing for
some degeneracy; see \cite{Jue16} for details.

Under some conditions,
model \eqref{1.eq1}, \eqref{1.AB} can be derived formally from a master equation
for a continuous-time, discrete-space random walk in the macroscopic limit
\cite{Ost11,ZaJu17} or from a fluiddynamical model in the inertia approximation
\cite[Section 4.2]{Jue16}. The variables $u_i$
may describe the density of the $i$th population species or the $i$th component
of a gas mixture with electrically charged components. In the former case,
$\phi$ models the environmental potential,
in the latter case, it denotes the electric potential. Because of these
applications, it is reasonable to assume that $u_i\ge 0$ in $\Omega$, $t>0$.

For special choices of $A_{ij}$ and $B_{ij}$, including condition \eqref{1.AB},
the existence of global bounded weak solutions can be shown. We give some examples
and references in section \ref{sec.ex} below. In this paper, we are only concerned
with the uniqueness of weak solutions.

\subsection{State of the art}

Before stating and explaining our assumptions and the main result, let us review
some techniques to show the uniqueness of (weak) solutions to
nonlinear parabolic equations. We focus on generalized solutions since
uniqueness of strong solutions is usually proved by standard $L^{2}$ estimations.

One important technique is based on the use of the test function
$\mbox{sign}_+(u^{(1)}-u^{(2)})$, where $u^{(1)}$ and $u^{(2)}$ are two solutions
and $\mbox{sign}_+$
is the positive sign function ($\mbox{sign}_+(s)=1$ for $s>0$ and $\mbox{sign}_+(s)=0$
else). The use of this test function can be justified by employing the
technique of doubling the variables, first developed by Kru\v{z}kov for
hyperbolic equations \cite{Kru70} and later extended by Carrillo
to scalar parabolic equations \cite{Car99} and by Blanchard and Porretta
to allow for renormalized solutions \cite{BlPo05}.
We refer to the review \cite{AnIg12} for
an extensive bibliography. All these results hold for scalar equations only.

Nonlinear semigroup methods provide powerful abstract tools for proving the
uniqueness of (mild of integral) solutions; see, e.g., \cite{BeWi96}.
However, this approach seems to be generally not accessible to cross-diffusion
systems.

One of the first uniqueness theorems for diffusion systems was shown by Alt and
Luckhaus \cite{AlLu83} under the assumptions that the time derivative of $u_i$
is integrable and the elliptic operator is linear. The first hypothesis was relaxed to
finite-energy solutions by Otto \cite{Ott96}, and the ellipticity condition
was generalized by Agueh using methods from optimal transport \cite{Agu05},
but in both cases for scalar equations only.

Another powerful approach is the dual method which consists in choosing a test function
which satisfies an appropriate dual problem \cite{DGJ01}.
This includes the $H^{-1}$ method,
where a test function of an {\em elliptic} dual problem is chosen.
In some sense, the uniqueness problem is reduced to an existence problem
of the dual problem \cite{Mas03}. The dual method allows one to treat
diffusion systems that are, to some extent, weakly coupled; see, e.g.,
\cite{DGJ01,Jue00,MiSu14}. Based on a dual method, Pham and Temam \cite{PhTe17} proved
recently a uniqueness result for a strongly coupled population system assuming a strictly
positive definite diffusion matrix.

The uniqueness of (weak) solutions may be also proven by using an entropy method.
One idea is to differentiate the relative entropy $H(u^{(1)}|u^{(2)})$, where
$u^{(1)}$ and $u^{(2)}$ are two solutions emanating from the same initial data,
with respect to time and to show that $\frac{d}{dt}H(u^{(1)}|u^{(2)})
\le CH(u^{(1)}|u^{(2)})$ for some constant $C>0$, which implies from Gronwall's
lemma that $H(u^{(1)}|u^{(2)})=0$ and hence $u^{(1)}=u^{(2)}$.
This approach has been used to show the weak-strong uniqueness
for compressible Navier-Stokes equations \cite{FJN12,FeNo12} and reaction-diffusion
systems (with diagonal diffusion matrix) \cite{Fis17}.
A second idea, due to Gajewski \cite{Gaj94}, is to
time-differentiate the semimetric
\begin{equation}\label{1.semimetric}
  d(u^{(1)},u^{(2)}) = H(u^{(1)}) + H(u^{(2)}) - 2H\bigg(\frac{u^{(1)}+u^{(2)}}{2}\bigg)
\end{equation}
for convex entropies $H$ and to show that
$\frac{d}{dt}d(u^{(1)},u^{(2)})\le 0$, implying again
that $u^{(1)}=u^{(2)}$. The technique has been applied to nonlinear drift-diffusion
equations for semiconductors \cite{Gaj94} and later to cross-diffusion systems
\cite{JuZa16,ZaJu17}. Compared to other methods, it has the advantage that
only weak solutions are needed \cite[Chapter 4.7]{Jue16}. The Gajewski method
is related to the approach of using relative entropies; see Remark \ref{rem.comp}.

\subsection{Assumptions and main result}

Our approach is to combine the $H^{-1}$ technique and the method of Gajewski
and to generalize the results from \cite{JuZa16,ZaJu17}.
The novelty is the inclusion of the potential term and the general structure
of $A_{ij}(u)$. With hypothesis \eqref{1.AB}, equations \eqref{1.eq1} can be formulated
as
\begin{equation}\label{1.ui}
  \pa_t u_i = \diver\big(p(u_0)\na u_i + q(u_0)u_i\na u_0 + r(u_0)u_i\na\phi\big),
	\quad i=1,\ldots,n.
\end{equation}
This can be interpreted as a drift-diffusion equation with field term
$q(u_0)\na u_0+r(u_0)\na\phi$. Since $u_0$ depends on all $u_i$, this is still
a cross-diffusion system. However, the drift-diffusion structure is essential
in the uniqueness proof. Our main result is as follows.

\begin{theorem}[Uniqueness of weak solutions]\label{thm1}
Let $(u,\phi)$ with $u=(u_1,$ $\ldots,u_n)$
be a weak solution to \eqref{1.eq1}-\eqref{1.bic}
such that $u_0(x,t)\in[0,L]$ for $x\in\Omega$, $t\in(0,T)$ and some $L>0$. Let
$u^0\in L^\infty(\Omega)$ and $f\in L^{2}(\Omega)$.
We assume that there exists $M>0$ such that for all $s\in[0,L]$,
\begin{align}
  & p(s)\ge 0, \quad p(s)+q(s)s\ge 0, \label{1.cond1} \\
	& r(s)s\in C^{1}([0,L]), \quad \frac{(r(s)+r'(s)s)^{2}}{p(s)+q(s)s}\le M.
	\label{1.cond2}
\end{align}
Then $(u,\phi)$ is unique in the class of solutions satisfying $\int_\Omega\phi dx=0$,
$\na\phi\in L^\infty(0,T;$ $L^\infty(\Omega))$, and
$$
  u_i\in L^{2}(0,T;H^{1}(\Omega)), \quad \pa_t u_i\in L^{2}(0,T;H^{1}(\Omega)'),
	\quad i=1,\ldots,n.
$$
In the case $r\equiv 0$, the boundedness of $u_0$ is not needed, provided that
\begin{equation}\label{1.cond3}
  \sqrt{p(u_0)}\nabla u_i,\ \sqrt{|q(u_0)|}\nabla u_i\in L^2(\Omega\times(0,T)).
\end{equation}

\end{theorem}

\begin{remark}\rm
1.\ The regularity assumption on the potential can be relaxed to $\na\phi\in
L^\infty(0,T;L^\alpha(\Omega))$ for $\alpha>d$ if $p(s)+q(s)s=\mbox{const.}>0$;
see Remark \ref{rem}.

2.\ If $\pa\Omega\in C^{1,1}$ and $f\in L^\alpha(\Omega)$ with $\alpha>d$,
the regularity $\na\phi\in L^\infty(0,T;L^\infty(\Omega))$ is a consequence
of elliptic theory. Indeed, since $u_0$ is bounded,
$u_0-f\in L^\infty(0,T;L^\alpha(\Omega))$, which implies, by Sobolev embedding,
that $\phi\in
L^\infty(0,T;W^{2,\alpha}(\Omega))\hookrightarrow L^\infty(0,T;W^{1,\infty}(\Omega))$.

\qed
\end{remark}

The idea of the proof is first to show the uniqueness of $(u_0,\phi)$.
Indeed, multiplying \eqref{1.eq1} by $a_i$ and summing over $i=1,\ldots,n$ leads to
a nonlinear drift-diffusion equation for $u_0$,
\begin{equation}\label{1.u0}
  \pa_t u_0 = \diver\big(\na Q(u_0) + R(u_0)\na\phi\big),
\end{equation}
coupled with the Poisson equation \eqref{1.eq2}, where
\begin{equation}\label{1.RQ}
  R(s) = r(s)s, \quad Q(s) = \int_0^s (p(\tau)+q(\tau)\tau)d\tau.
\end{equation}
Since the diffusion operator in \eqref{1.u0} may degenerate, it is natural
to apply the $H^{-1}$ technique. Indeed,
given two solutions $(u^{(1)},\phi^{(1)})$, $(u^{(2)},\phi^{(2)})$with the same
initial data, we use the test function
$\phi:=\phi^{(1)}-\phi^{(2)}$ in \eqref{2.eq1}, which solves the dual problem
$-\Delta\phi=u^{(1)}-u^{(2)}$ in $\Omega$, $\na\phi\cdot\nu=0$ on $\pa\Omega$.
Then, using conditions \eqref{1.cond1}-\eqref{1.cond2}, it can be shown that
$\frac{d}{dt}\|\na\phi\|_{L^{2}(\Omega)}^{2}\le C\|\na\phi\|_{L^{2}(\Omega)}^{2}$,
which implies
that $u^{(1)}=u^{(2)}$ and $\phi^{(1)}=\phi^{(2)}$. In this step,
we need the regularity $\na\phi^{(2)}\in L^\infty$.

The second step is to prove the uniqueness of \eqref{1.ui}. For this, we employ
the method of Gajewski \cite{Gaj94}, based on an estimation of the semimetric
$$
  d(u,v) = \sum_{i=1}^n\int_\Omega\bigg(h(u_i)+h(v_i)-2h\bigg(\frac{u_i+v_i}{2}\bigg)
	\bigg)dx,
$$
where $h(s)=s(\log s-1)+1$ is called an entropy.
Let $u^{(1)}=(u_1^{(1)},\ldots,u_n^{(1)})$, $u^{(2)}=(u_1^{(2)},\ldots,u_n^{(2)})$
be two weak solutions
to \eqref{1.ui}. A formal computation shows that
$\frac{d}{dt}d(u^{(1)},$ $u^{(2)})\le 0$ and hence $d(u^{(1)},u^{(2)})=0$.
The convexity of $h$ implies that $u^{(1)}=u^{(2)}$. In order to make this
argument rigorous, we need to regularize the entropy, since terms with $\log u_i$
may be not defined on sets where $u_i=0$. We discuss in Remark \ref{rem.comp}
the applicability of the Gajewski method.

The paper is organized as follows. Theorem \ref{thm1} is proved in
section \ref{sec.proof}. Section \ref{sec.rem} is concerned with some comments
on the techniques and the proof. Some examples satisfying conditions
\eqref{1.cond1}-\eqref{1.cond2} are detailed in section \ref{sec.ex}.


\section{Proof of Theorem \ref{thm1}}\label{sec.proof}

{\em Step 1. Uniqueness for $(u_0,\phi)$.} We multiply \eqref{1.eq1} by $a_i$ and sum
over $i=1,\ldots,n$:
\begin{align}
  \pa_t u_0 &= \sum_{i,j=1}^n\diver\Big(\delta_{ij} p(u_0)a_i\na u_j
	+ q(u_0)a_iu_i\na(a_ju_j) + \delta_{ij} r(u_0)a_iu_i\na\phi\Big) \nonumber \\
	&= \diver\big(p(u_0)\na u_0 + q(u_0)u_0\na u_0 + r(u_0)u_0\na\phi\big) \nonumber \\
	&= \diver\big(\na Q(u_0) + R(u_0)\na\phi\big), \label{2.eq1}
\end{align}
where $Q$ and $R$ are defined in \eqref{1.RQ}. Clearly, it holds
$$
  \na Q(u_0)\cdot\nu=\sum_{i=1}^na_i\sum_{j=1}^nA_{ij}\nabla u_j\cdot\nu=0\quad
	\mbox{on }\pa\Omega.
$$
In view of condition \eqref{1.cond1}, the function $Q$ is nondecreasing.
We use the $H^{-1}$ method to prove that \eqref{1.eq2}, \eqref{2.eq1} possesses
at most one solution. Let $(u_0^{(1)},\phi^{(1)})$ and $(u_0^{(2)},\phi^{(2)})$
be two solutions
to \eqref{1.eq2}, \eqref{2.eq1}, subject to no-flux boundary conditions and
the same initial condition \eqref{1.bic}. We set $u_0=u_0^{(1)}-u_0^{(2)}$ and
$\phi=\phi^{(1)}-\phi^{(2)}$. It holds that $\int_\Omega u_0dx=0$, and $\phi$ solves
$$
  -\Delta\phi = u_0\quad\mbox{in }\Omega, \quad
	\na\phi\cdot\nu=0\quad\mbox{on }\pa\Omega, \quad \int_\Omega\phi dx=0.
$$
Since $u_0\in L^{2}(\Omega\times (0,T))$ and $\pa_t u_0\in L^{2}(0,T;H^{1}(\Omega)')$,
we have $\phi\in L^{2}(0,T;H^{2}(\Omega))$ and
$\pa_t \Delta\phi\in L^{2}(0,T;H^{1}(\Omega)')$. By applying a standard mollification
procedure, we can prove that $t\mapsto \|\na\phi(t)\|_{L^{2}(\Omega)}^{2}$ is
continuous on $[0,T]$ (possibly after redefinition on a set of measure zero) and
$$
  \frac12\frac{d}{dt}\|\na\phi(t)\|_{L^{2}(\Omega)}^{2}
	= -\langle\pa_t\Delta\phi(t),\phi(t)\rangle = \langle\pa_t u_0(t),\phi(t)\rangle,
$$
where $\langle\cdot,\cdot\rangle$ is the dual product between $H^{1}(\Omega)'$ and
$H^{1}(\Omega)$. Observe that at time $t=0$, $-\Delta\phi(0)=0$ and hence,
$\phi(0)=0$. Using $\phi$ as a test function in the difference of the weak
formulations of \eqref{2.eq1} for $u_0^{(1)}$ and $u_0^{(2)}$, respectively,
it follows that
\begin{align}
  \frac12\|\na\phi(t)\|_{L^{2}(\Omega)}^{2}
	&= -\int_0^t\int_\Omega\na\big(Q(u_0^{(1)})-Q(u_0^{(2)})\big)\cdot\na\phi
	dxds \nonumber \\
	&\phantom{xx}{}- \int_0^t\int_\Omega\big(R(u_0^{(1)})\na\phi^{(1)}-R(u_0^{(2)})
	\na\phi^{(2)}\big)\cdot\na\phi dxds \nonumber \\
	&= -\int_0^t\int_\Omega\big(Q(u_0^{(1)})-Q(u_0^{(2)})\big)(u_0^{(1)}-u_0^{(2)})dxds
	\nonumber \\
	&\phantom{xx}{}- \int_0^t\int_\Omega R(u_0^{(1)})|\na\phi|^{(2)} dxds \nonumber \\
	&\phantom{xx}{}-\int_0^t\int_\Omega\big(R(u_0^{(1)})-R(u_0^{(2)})\big)
	\na\phi^{(2)}\cdot\na\phi dxds. \label{2.aux}
\end{align}

The second integral on the right-hand side is estimated as
$$
  \bigg|\int_0^t\int_\Omega R(u_0^{(1)})|\na\phi|^{2} dxds\bigg|
	\le \|R(u_0^{(1)})\|_{L^\infty(0,T;L^\infty(\Omega))}\|\na\phi\|_{L^{2}(Q_T)}^{2}
	\le C_1\|\na\phi\|_{L^{2}(Q_T)}^{2}
$$
for some constant $C_1>0$, where $Q_T=\Omega\times(0,T)$.
This is the only place where the boundedness of $u_0$ is needed.

To estimate the last integral in \eqref{2.aux},
we use the assumption $\na\phi^{(2)}\in L^\infty(0,T;L^\infty(\Omega))$ and
Young's inequality with $\eps>0$:
\begin{align}
  \bigg|\int_0^t\int_\Omega & \big(R(u_0^{(1)})-R(u_0^{(2)})\big)\na\phi^{(2)}
	\cdot\na\phi dxds\bigg|
	\le C_2\int_0^t\int_\Omega\big|R(u_0^{(1)})-R(u_0^{(2)})\big||\na\phi|dxds \nonumber \\
	&\le \int_0^t\int_\Omega\big((Q(u_0^{(1)})-Q(u_0^{(2)}))u_0+\eps\big)dxds \nonumber \\
	&\phantom{xx}{}+ \frac{C_2^{2}}{4}\int_0^t\int_\Omega
	\frac{(R(u_0^{(1)})-R(u_0^{(2)}))^{2}}{(Q(u_0^{(1)})-Q(u_0^{(2)}))u_0+\eps}
	|\na\phi|^{2} dxds.
	\label{2.aux2}
\end{align}
We claim that the quotient is bounded. Indeed, by assumption \eqref{1.cond2},
$(R')^{2}/Q'$ is bounded on $[0,L]$ and hence, by H\"older's inequality,
\begin{align*}
  \bigg(\int_0^{1} R'(\theta u_0^{(1)} + (1-\theta)u_0^{(2)})d\theta\bigg)^{2}
	&\le \int_0^{1}\frac{(R'(\theta u_0^{(1)} + (1-\theta)u_0^{(2)}))^{(2)}}{
	Q'(\theta u_0^{1} + (1-\theta)u_0^{(2)})}d\theta \\
	&\phantom{xx}{}\times\int_0^{1} Q'(\theta u_0^{1} + (1-\theta)u_0^{(2)})d\theta \\
	&\le C_3\int_0^{1}Q'(\theta u_0^{(1)} + (1-\theta)u_0^{(2)})d\theta.
\end{align*}
This shows that
$$
  \frac{(R(u_0^{(1)})-R(u_0^{(2)}))^{2}}{(Q(u_0^{(1)})-Q(u_0^{(2)}))u_0+\eps}
	= \displaystyle\frac{\big(\int_0^{1} R'(\theta u_0^{(1)} + (1-\theta)u_0^{(2)})d\theta
	\big)^{2}u_0^{2}}{\int_0^{1} Q'(\theta u_0^{(1)}
	+ (1-\theta)u_0^{(2)})d\theta u_0^{2}+\eps} \\
  \le C_3.
$$
Then \eqref{2.aux2} becomes
\begin{align*}
  \bigg|\int_0^t\int_\Omega\big(R(u_0^{(1)})-R(u_0^{(2)})\big)\na\phi^{(2)}\cdot\na\phi
	dxds\bigg|
	&\le \int_0^t\int_\Omega\big((Q(u_0^{(1)})-Q(u_0^{(2)}))u_0+\eps\big)dxds \\
	&\phantom{xx}{}+ \frac14 C_2^{2}C_3\int_0^t\int_\Omega|\na\phi|^{2} dxds.
\end{align*}
In the limit $\eps\to 0$, we obtain
\begin{align*}
  \bigg|\int_0^t\int_\Omega\big(R(u_0^{(1)})-R(u_0^{(2)})\big)\na\phi^{(2)}\cdot\na\phi
	dxds\bigg|
	&\le \int_0^t\int_\Omega(Q(u_0^{(1)})-Q(u_0^{(2)}))u_0 dxds \\
	&\phantom{}+ \frac14 C_2^{2}C_3\|\na\phi\|_{L^{2}(Q_T)}^{2}.
\end{align*}
The first integral on the right-hand side is absorbed by the first integral on
the right-hand side of \eqref{2.aux}, and we end up with
$$
  \|\na\phi(t)\|_{L^{2}(\Omega)}^{2} \le C_4\int_0^t\|\na\phi\|_{L^{2}(\Omega)}^{2}ds,
$$
where $C_4=2C_1+C_2^{2}C_3/2$. Finally, by Gronwall's lemma, it follows that
$\na\phi(t)=0$ in $\Omega$ and $u_0(t)=-\Delta\phi(t)=0$. Since
$\int_\Omega\phi(t)dx=0$, we also have $\phi(t)=0$ for $t\in(0,T)$.
This shows that \eqref{1.eq2}, \eqref{2.eq1} is uniquely solvable.

{\em Step 2. Uniqueness for $u_i$.} Let $u^{(1)}=(u_1^{(1)},\ldots,u_n^{(1)})$ and
$u^{(2)}=(u_1^{(2)},\ldots,u_n^{(2)})$ be two weak solutions to \eqref{1.eq1}.
In this step, the solutions are not required to be bounded.
We set $u_0^{(1)}=\sum_{i=1}^n a_iu_i^{(1)}$, $u_0^{(2)}=\sum_{i=1}^n a_iu_i^{(2)}$.
Step 1 shows that $u_0:=u_0^{(1)}=u_0^{(2)}$ and the corresponding potential $\phi$
is unique. Then $u^{(1)}$ and $u^{(2)}$ solve, respectively,
\begin{equation}\label{2.uij}
  \pa_t u_i^{(j)} = \diver\big(p(u_0)\na u_i^{(j)} + u_i^{(j)} F\big), \quad j=1,2,
\end{equation}
with corresponding no-flux and initial conditions, where
$F=q(u_0)\na u_0+r(u_0)\na\phi\in L^2(Q_T)$. Let $0<\eps<1$.
We introduce, as in \cite{ZaJu17}, the regularized entropy
$$
  h_\eps(s) = (s+\eps)\big(\log(s+\eps)-1\big) + 1, \quad s\ge 0,
$$
and the semimetric
$$
  d_\eps(u,v) = \sum_{i=1}^n\int_\Omega\bigg(h_\eps(u_i)+h_\eps(v_i)
	-2h_\eps\bigg(\frac{u_i+v_i}{2}\bigg)\bigg)dx
$$
for appropriate functions $u=(u_1,\ldots,u_n)$, $v=(v_1,\ldots,v_n)$.
Since $h_\eps$ is convex, we have $d_\eps(u,v)\ge 0$.

We recall the following result. Let $0\le w\in L^{2}(0,T;H^{1}(\Omega))\cap
H^{1}(0,T;H^{1}(\Omega)')$. Then $t\mapsto \int_\Omega h_\eps(w(t))dx$ is
absolutely continuous and
$$
  \frac{d}{dt}\int_\Omega h_\eps(w(t))dx = \langle \pa_t w,\log(w+\eps)\rangle.
$$
Therefore, we can differentiate
$t\mapsto d_\eps(u^{(1)}(t),u^{(2)}(t))$, yielding
\begin{align*}
  \frac{d}{dt}d_\eps(u^{(1)},u^{(2)})
	&= \sum_{i=1}^n\bigg(\langle\pa_t u_i^{(1)},\log(u_i^{(1)}+\eps)\rangle
	+ \langle\pa_t u_i^{(2)},\log(u_i^{(2)}+\eps)\rangle \\
	&\phantom{xx}{}- \bigg\langle\pa_t(u_i^{(1)}+u_i^{(2)}),
	\log\bigg(\frac{u_i^{(1)}+u_i^{(2)}}{2}+\eps\bigg)\bigg\rangle\bigg) \\
	&= -\sum_{i=1}^n\int_\Omega\big(p(u_0)\na u_i^{(1)}+u_i^{(1)}F\big)
	\cdot\frac{\na u_i^{(1)}}{u_i^{(1)}+\eps}dx \\
	&\phantom{xx}{}- \sum_{i=1}^n\int_\Omega\big(p(u_0)\na u_i^{(2)}+u_i^{(2)}F\big)
	\cdot\frac{\na u_i^{(2)}}{u_i^{(2)}+\eps}dx \\
  &\phantom{xx}{}+ \sum_{i=1}^n\int_\Omega\big(p(u_0)\na(u_i^{(1)}+u_i^{(2)})
	+ (u_i^{(1)}+u_i^{(2)})F\big)
	\frac{\na(u_i^{(1)}+u_i^{(2)})}{u_i^{(1)}+u_i^{(2)}+2\eps}dx.
\end{align*}
Rearranging the terms, we end up with
\begin{align*}
  \frac{d}{dt}d_\eps(u^{(1)},u^{(2)})
	&= -\sum_{i=1}^n\int_\Omega p(u_0)
	\bigg(\frac{|\na u_i^{(1)}|^{2}}{u_i^{(1)}+\eps}
	+ \frac{|\na u_i^{(2)}|^{2}}{u_i^{(2)}+\eps}
	- \frac{|\na(u_i^{(1)}+u_i^{(2)})|^{2}}{u_i^{(1)}+u_i^{(2)}+2\eps}\bigg)dx \\
	&\phantom{xx}{}- \sum_{i=1}^n\int_\Omega F
	\cdot\na u_i^{(1)}\bigg(\frac{u_i^{(1)}}{u_i^{(1)}+\eps}
	- \frac{u_i^{(1)}+u_i^{(2)}}{u_i^{(1)}+u_i^{(2)}+2\eps}\bigg)dx \\
	&\phantom{xx}{}- \sum_{i=1}^n\int_\Omega F
	\cdot\na u_i^{(2)}\bigg(\frac{u_i^{(2)}}{u_i^{(2)}+\eps}
	- \frac{u_i^{(1)}+u_i^{(2)}}{u_i^{(1)}+u_i^{(2)}+2\eps}\bigg)dx.
\end{align*}
Since for suitable functions $u$, $v$,
$$
  \frac{|\na u|^{2}}{u+\eps}	+ \frac{|\na v|^{2}}{v+\eps}
	- \frac{|\na(u+v)|^{2}}{u+v+2\eps} \\
	= \frac{1}{u+v+2\eps}\bigg|\sqrt{\frac{v+\eps}{u+\eps}}\na u
	- \sqrt{\frac{u+\eps}{v+\eps}}\na v\bigg|^{2},
$$
the first term is nonpositive. Then, integrating in time and observing that
$d_\eps(u^{(1)}(0),$ $u^{(2)}(0))$ $=0$, it follows that
\begin{align}
  d_\eps (u^{(1)}(t),u^{(2)}(t))
	&\le - \sum_{i=1}^n\int_0^t\int_\Omega F
	\cdot\na u_i^{(1)}\bigg(\frac{u_i^{(1)}}{u_i^{(1)}+\eps}
	- \frac{u_i^{(1)}+u_i^{(2)}}{u_i^{(1)}+u_i^{(2)}+2\eps}\bigg)dxds \nonumber \\
	&\phantom{xx}{}- \sum_{i=1}^n\int_0^t\int_\Omega F
	\cdot\na u_i^{(2)}\bigg(\frac{u_i^{(2)}}{u_i^{(2)}+\eps}
	- \frac{u_i^{(1)}+u_i^{(2)}}{u_i^{(1)}+u_i^{(2)}+2\eps}\bigg)dxds. \label{2.aux3}
\end{align}
Expanding $h_\eps(u_i^{(1)})$ and $h_\eps(u_i^{(2)})$ at $(u_i^{(1)}+u_i^{(2)})/2$
up to second order and summing the resulting expressions,
we find that, for some $\theta_i^{(k)}\in (0,1)$ $(k=1,2)$ and
$\xi_i^{(k)}=\theta_i^{(k)}u_i^{(k)}+(1-\theta_i^{(k)})(u_i^{(1)}+u_i^{(2)})/2$,
\begin{align*}
  d_\eps(u^{(1)},u^{(2)})
  &=\frac18\sum_{i=1}^n\int_\Omega\big(h''_\eps(\xi_i^{(1)})
  +h''_\eps(\xi_i^{(1)})\big) (u_i^{(1)}-u_i^{(2)})^{2}dx \\
  \ge &\frac14\sum_{i=1}^n\int_\Omega
	\frac{(u_i^{(1)}-u_i^{(2)})^{2}}{\max\{u_i^{(1)},u_i^{(2)}\}+\eps}dx
	\ge \frac14\sum_{i=1}^n\int_\Omega
  \frac{(u_i^{(1)}-u_i^{(2)})^{2}}{\max\{u_i^{(1)},u_i^{(2)}\}+1}dx.
\end{align*}
Since $F\cdot\na u_i^j\in L^{1}(Q_T)$ for $j=1,2$, we
may apply the dominated convergence theorem giving, as $\eps\to 0$,
$$
  \int_0^t\int_\Omega F
	\cdot\na u_i^j\bigg(\frac{u_i^j}{u_i^j+\eps}
	- \frac{u_i^{(1)}+u_i^{(2)}}{u_i^{(1)}+u_i^{(2)}+2\eps}
	\bigg)dxds \to 0, \quad j=1,2.
$$
Therefore, \eqref{2.aux3} becomes
$$
  0 \le \int_\Omega\frac{(u_i^{(1)}-u_i^{(2)})(t)^2}{
	\max\{u_i^{(1)}(t),u_i^{(2)}(t)\}+1}dx = 0,
$$
and thus, $u_i^{(1)}(t)=u_i^{(2)}(t)=0$ for $t\in(0,T)$ since $u_i^{(j)}(t)$ is
finite a.e.\ in $\Omega$.

If $r\equiv0$ and $u_0$ is not bounded, then we need the integrability 
\eqref{1.cond3} to make the computations rigorous.
This concludes the proof of Theorem \ref{thm1}.


\section{Remarks}\label{sec.rem}

We give two comments on the regularity of the drift term and on the
relation of Gajewski's semimetric to relative entropies.

\begin{remark}[Lower regularity of $\na\phi$]\label{rem}\rm
We claim that the regularity on $\phi$ can be relaxed to
$\na\phi\in L^\infty(0,T;L^\alpha(\Omega))$ with $\alpha>d$ if
$p(s)+q(s)s=D=\mbox{const.}>0$. For simplicity, we assume that $D=1$. 
In this case, we do not need to apply the
$H^{-1}$ method and can use standard $L^{2}$ estimates. Let $(u^{(1)},\phi^{(1)})$ and
$(u^{(2)},\phi^{(2)})$ be two solutions to \eqref{1.eq2}, \eqref{2.eq1} with
the same boundary and initial conditions. Taking $u^{(1)}-u^{(2)}$
as a test function
in \eqref{2.eq1}, we find that
\begin{align}
  \frac12 & \int_\Omega(u_0^{(1)}-u_0^{(2)})^{2}(t) dx
	+ \int_0^t\int_\Omega|\na(u_0^{(1)}-u_0^{(2)})|^{2} dxds \nonumber \\
	&= -\int_0^t\int_\Omega (R(u_0^{(1)})-R(u_0^{(2)}))\na\phi^{(1)}
	\cdot\na(u_0^{(1)}-u_0^{(2)})dxds \nonumber \\
	&\phantom{xx}{}- \int_0^t\int_\Omega R(u_0^{(2)})\na(\phi^{(1)}-\phi^{(2)})
	\cdot\na(u_0^{(1)}-u_0^{(2)})dxds \nonumber \\
	&=: I_1 + I_2. \label{2.aux4}
\end{align}
By the boundedness of $u_i^{(2)}$ and the elliptic estimate for the Poisson equation,
the second integral is estimated as
\begin{align*}
  I_2 &\le C_5\|\na(\phi^{(1)}-\phi^{(2)})\|_{L^{2}(Q_t)}
	\|\na(u_0^{(1)}-u_0^{(2)})\|_{L^{2}(Q_t)} \\
	&\le \frac14\|\na(u_0^{(1)}-u_0^{(2)})\|_{L^{2}(Q_t)}^{2}
	+ C_6\|u_0^{(1)}-u_0^{(2)}\|_{L^{2}(Q_t)}^{2},
\end{align*}
where $Q_t=\Omega\times(0,t)$ and $C_5>0$ depends on the $L^\infty$ norm of $u_0^{(2)}$.
For the first integral $I_1$, we employ the Lipschitz continuity of $R$,
he Cauchy-Schwarz inequality, the H\"older inequality,
the Gagliardo-Nirenberg inequality with $\theta=d/2-d/\beta\in(0,1)$,
and eventually the Young inequality with parameter $\theta$:
\begin{align*}
  I_1 &\le \frac14\|\na(u_0^{(1)}-u_0^{(2)})\|_{L^{2}(Q_t)}^{2}
	+ C_7\int_0^t\|u_0^{(1)}-u_0^{(2)}\|_{L^\beta(\Omega)}^{2}
	\|\na\phi^{(1)}\|_{L^\alpha(\Omega)}^{2}ds \\
	&\le \frac14\|\na(u_0^{(1)}-u_0^{(2)})\|^{2}_{L^{2}(Q_t)}
    + C_8\big(\|\na\phi^{(1)}\|_{L^\infty(0,T;L^\alpha(\Omega)}\big) \\
  &\phantom{xx}{}\times\int_0^t\big(\|\na(u_0^{(1)}-u_0^{(2)})
	\|_{L^{2}(\Omega)}^{2\theta}
  \|u_0^{(1)}-u_0^{(2)}\|_{L^{2}(\Omega)}^{2(1-\theta)}
	+ \|u_0^{(1)}-u_0^{(2)}\|_{L^{2}(\Omega)}^{2}\big)ds \\
	&\le \frac12\|\na(u_0^{(1)}-u_0^{(2)})\|_{L^{2}(Q_t)}^{2}
	+ C_9\big(\|\na\phi^{(1)}\|_{L^\infty(0,T;L^\alpha(\Omega)}\big)
	\int_0^t\|u_0^{(1)}-u_0^{(2)}\|_{L^{2}(\Omega)}^{2} ds.
\end{align*}
Therefore, \eqref{2.aux} becomes
$$
  \|(u_0^{(1)}-u_0^{(2)})(t)\|_{L^{2}(\Omega)}^{2}
	\le C_{10}\int_0^t\|u_0^{(1)}-u_0^{(2)}\|_{L^{2}(\Omega)}^{2}ds,
$$
and Gronwall's lemma shows that $(u_0^{(1)}-u_0^{(2)})(t)=0$ in $\Omega$, $t>0$.
\qed
\end{remark}

\begin{remark}[Comparison of Gajewski's semimetric and relative entropies]\
\label{rem.comp}\rm
In the sec\-ond step of the proof of Theorem \ref{thm1}, we may work with
another semimetric, based on the relative entropy
$$
  H(u|v) = H(u) - H(v) - H'(v)\cdot(u-v),
$$
as done in \cite{Fis17}, where $H(u)=\sum_{i=1}^n\int_\Omega h(u_i)dx$ with
$h(u_i)=u_i(\log u_i-1)+1$.
Setting $h(u)=(h(u_1),...,h(u_n))$ by a slight abuse of notation,
we see that $h:\R^n\to\R^n$ is a convex function.
Instead of the expression from \cite{Fis17}, we
use its symmetrized version to obtain a semimetric:
\begin{equation}\label{d0}
  d_0(u,v) = H(u|v) + H(v|u) = \int_\Omega(h'(u)-h'(v))\cdot(u-v)dx.
\end{equation}
The semimetrics \eqref{1.semimetric} and \eqref{d0} are strongly related although
they are different. First, both expressions behave like $|u-v|^2$ for ``small''
$|u-v|$, since a Taylor expansion shows that both semimetrics can
be estimated from below by, up to a factor, $(u-v)^\top h''(\xi)(u-v)$, where
$h''(\xi)$ is the Hessian of $h$ at some point $\xi\in\R^n$. Second, when
differentiating $d_0(u^{(1)},u^{(2)})$ with respect to time and inserting
\eqref{1.eq1}, the drift terms cancel, as they do when differentiating
$d(u^{(1)},u^{(2)})$. A formal computation shows that
$$
  \frac{d}{dt}d_0(u^{(1)},u^{(2)})
	= -\sum_{i=1}^n\int_\Omega p(u_0)(u_i+v_i)\bigg|\na\log\frac{u_i}{v_i}\bigg|^2 dx
	\le 0,
$$
implying that $u^{(1)}=u^{(2)}$. In order to make this argument rigorous,
we need to work as in section \ref{sec.proof} with a regularization
(replacing $u_i^{(j)}$ by $u_i^{(j)}+\eps$).

In fact, the previous argument can be generalized to 
the following family of semimetrics. Let $d_1(u,v)=\int_\Omega g(u,v)dx$
for some smooth symmetric convex function $g$ and let $u^{(1)}$ and $u^{(2)}$
be two solutions to the scalar equation
\begin{equation}\label{3.u}
  \pa_t u = \diver(a(x)\na u + uF(x)),
\end{equation}
which resembles \eqref{2.uij}, with no-flux boundary conditions and the same
initial condition. We assume that $a(x)\ge 0$ and $F(x)\in\R^n$. Set
$$
  g_{11}=\frac{\pa^2 g}{\pa u^2}(u^{(1)},u^{(2)}), \quad
  g_{12}=\frac{\pa^2 g}{\pa u\pa v}(u^{(1)},u^{(2)}), \quad
  g_{22}=\frac{\pa^2 g}{\pa v^2}(u^{(1)},u^{(2)}).
$$	
Then, formally,
\begin{align*}
  \frac{d}{dt}d_1(u^{(1)},u^{(2)})
	&= -\int_\Omega a(x)\big(g_{11}|\na u^{(1)}|^2
	+ 2g_{12}\na u^{(1)}\cdot\na u^{(2)} + g_{22}|\na u^{(2)}|^2\big)dx \\
	&\phantom{xx}{}- \int_\Omega F(x)\cdot\big((u^{(1)}g_{11}+u^{(2)}g_{12})\na u
	+ (u^{(1)}g_{12}+u^{(2)}g_{22})\na u^{(2)}\big)dx.
\end{align*}
Since $g$ is convex, the first integral is nonnegative. If we assume that
\begin{equation}\label{3.g}
  u\frac{\pa^2 g}{\pa u^2} + v\frac{\pa^2 g}{\pa u\pa v} = 0\quad
	\mbox{for all }u,v,
\end{equation}
then the second integral vanishes (using the symmetry of $g$) and consequently,
$\frac{d}{dt}d_1(u^{(1)},$ $u^{(2)})\le 0$, which implies that $u^{(1)}=u^{(2)}$.
The integrands of the semimetrics \eqref{1.semimetric} and \eqref{d0} satisfy
condition \eqref{3.g}.
This argument shows that the linearity in the diffusion term of \eqref{3.u}
is essential for the entropy method.
\qed
\end{remark}


\section{Examples}\label{sec.ex}

Theorem \ref{thm1} can be applied to some cross-diffusion systems arising in
applications.


\subsection{Maxwell-Stefan equations}

The first example are the Maxwell-Stefan equations \cite{Max66,Ste71}
\begin{equation}\label{3.MS}
  \pa_t u_i + \diver J_i = 0, \quad
	\na u_i = -\sum_{j=1,\,j\neq i}^{n+1}d_{ij}(u_jJ_i-u_iJ_j),
	\quad i=1,\ldots,n+1,
\end{equation}
where $J_i$ are the fluxes and $d_{ij}$ the diffusion coefficients.
For a formal derivation, see \cite[Section 4.2]{Jue16}.
We assume that the sum of all concentrations is constant, $\sum_{i=1}^{n+1}u_i=1$,
which implies that $\sum_{i=1}^{n+1}J_i=0$.
In contrast to \eqref{1.eq1}, the fluxes are not a linear combination of the
gradients $\na u_i$, and we need to invert the flux-gradient relations.
However, because of $\sum_{i=1}^{n+1}J_i=0$, the relations cannot be directly
inverted. One idea is to remove the variable $u_{n+1}=1-\sum_{i=1}^n$, ending up
with $n$ equations, formulated as $\na u'=A_0J'$ \cite{JuSt12},
where $u'=(u_1,\ldots,u_n)$,
$J'=(J_1,\ldots,J_n)$, and $A_0=(A_{ij}^0)\in\R^{n\times n}$ with
\begin{align*}
  A_{ij}^0 &= -(d_{ij}-d_{i,n+1})u_i, \quad i\neq j,\ i,j=1,\ldots,n,  \\
	A_{ii}^0 &= \sum_{j=1,\,j\neq i}^n(d_{ij}-d_{i,n+1})u_{j} + d_{i,n+1},
	\quad i=1,\ldots,n.
\end{align*}
is invertible. The existence of global bounded weak solutions was shown in
\cite{JuSt12}.

\begin{corollary}[Maxwell-Stefan model]
Let $d_{ij}=D_0$ and $d_{i,n+1}=D$ for $i,j=1,\ldots,n$. Then the Maxwell-Stefan
system \eqref{1.bic}, \eqref{3.MS} has at most one weak solution.
\end{corollary}

\begin{proof}
By assumption, we have
$$
  A_{ij}^0 = \delta_{ij}\bigg(D+(D_0-D)\sum_{k=1,\,k\neq i}^{n}u_k\bigg)
	-(1-\delta_{ij})(D_0-D)u_i.
$$
A computation shows that the inverse $A(u)=A_0^{-1}$ is given by
$$
  A_{ij}(u) = \frac{\delta_{ij}D+(D_0-D)u_i}{D^{2}+D(D_0-D)\sum_{k=1}^n u_i}.
$$
This expression is of the form \eqref{1.AB} with $a_i=1$ and
$$
  p(s) = \frac{D}{D^{2}+D(D_0-D)s},
	\quad q(s) = \frac{D_0-D}{D^{2}+D(D_0-D)s}.
$$
The assumptions of Theorem \ref{thm1} are satisfied since $r(s)=0$ and
$$
  p(s) \ge \frac{1}{\max\{D_0,D\}} > 0, \quad
	p(s) + q(s)s = \frac{1}{D} > 0.
$$
This concludes the proof.
\end{proof}


\subsection{Shigesada-Kawasaki-Teramoto equations}

The second example is the Shigesada-Kawasaki-Teramoto system \eqref{1.eq1} arising in
population dynamics \cite{SKT79} with coefficients
\begin{equation}\label{3.SKT}
  A_{ij}(u) = \delta_{ij}\bigg(a_{i0} + \sum_{j=1}^n a_{ij}u_j\bigg)
	+ a_{ij}u_i, \quad B_{ij}(u) = \delta_{ij}u_i, \quad i,j=1,\ldots,n,
\end{equation}
where $a_{ij}>0$ for $i=0,\ldots,n$, $j=1,\ldots,n$.
The variables $u_i$ model population densities of interacting species
subject to some environmental potential. A formal derivation was given in
\cite[Section 4.2]{Jue16}.
The existence of global weak solutions was proved in \cite{CDJ16} (with $B_{ij}=0$)
under the assumption that there exists a vector $(\pi_1,\ldots,\pi_n)$
such that the detailed-balance condition
$\pi_i a_{ij}=\pi_j a_{ij}$ for all $i,j=1,\ldots,n$ holds or the
self-diffusion $a_{ii}$ dominates cross-diffusion $a_{ij}$ ($i\neq j$).
Under additional conditions and for $n=2$, the weak solutions are bounded
\cite{JuZa16}.

\begin{corollary}[Population dynamics model]
Let $a_{i0}=a_0>0$ and $a_{ij}=a_j>0$ for $i,j=1,\ldots,n$.
Then \eqref{1.eq1}-\eqref{1.bic}, \eqref{3.SKT} has at most one bounded weak solution.
\end{corollary}

Note that under the conditions of the corollary,
the detailed-balance condition is satisfied with $\pi_i=a_i$.
The corollary follows from Theorem \ref{thm1} by setting $p(s)=a_0+s\ge a_0>0$,
$q(s)=1$, and $r(s)=1$.


\subsection{A volume-filling model for ion transport}

The ion-transport model is defined by
\begin{equation}\label{3.vf}
  A_{ij}(u) = D_iu_i\quad\mbox{for }i\neq j, \quad
	A_{ii}(u)=D_i(1-u_0+u_i), \quad B_{ij} = z_i(1-u_0)u_i\delta_{ij},
\end{equation}
where $u_0=\sum_{i=1}^n u_i$ and $D_i>0$, $z_i\in\R$ are some constants \cite{BSW12}.
The variables $u_i$ represent the ion concentraton of the $i$th species and
$u_{n+1}:=1-u_0$ the solvent concentration.
The model can be derived formally from a random-walk lattice model
\cite{Ost11,Jue16}. The existence of global bounded weak solutions
was shown in \cite{ZaJu17} without potential and in \cite{GeJu17} including
the potential term.
Formulation \eqref{1.AB} is obtained for $D_i=D>0$ and $z_i=z\in\R$
by setting $a_i=1$, $p_i(s)=D(1-s)$, $q_i(s)=D$, and $r_i(s)=z(1-s)$.
The following result was already proved in \cite{GeJu17}. We show here that
the model fits in our framework.

\begin{corollary}[Ion-transport model]
Let $D_i=D>0$ and $z_i=z\in\R$ for $i=1,\ldots,n$.
Then \eqref{1.eq1}-\eqref{1.bic}, \eqref{3.vf} has at most one bounded weak solution
with $\na\phi\in L^\infty(0,T;L^\alpha(\Omega))$ and $\alpha>d$.
\end{corollary}

\begin{proof}
Conditions \eqref{1.cond1}-\eqref{1.cond2} are satisfied since
$p(s)+q(s)s=D>0$ and $r(s)$ is continuous on $[0,1]$.
By Remark \ref{rem}, the uniqueness result holds for potentials satisfying
$\na\phi\in L^\infty(0,T;L^\alpha(\Omega))$ with $\alpha>d$.
\end{proof}


%


\end{document}